\documentclass[12pt]{amsart}
\usepackage{amssymb,amsthm,mathrsfs,amsmath,bm,cite}

\newtheorem{thm}{Theorem}[section]
\newtheorem{theorem}[thm]{Theorem}

\newtheorem{lemma}[thm]{Lemma}
\newtheorem{proposition}[thm]{Proposition}

\theoremstyle{definition}
\newtheorem{definition}[thm]{Definition}

\theoremstyle{remark}
\newtheorem{remark}[thm]{Remark}

\newenvironment{theorem*}[1]{\smallskip\noindent{\bf #1.}\it}{\medskip}

\numberwithin{equation}{section} 
\newcommand\tr{\operatorname{tr}}
\newcommand\dom{\operatorname{dom}}
\newcommand\Iso{\operatorname{Iso}}

\newcommand\bC{{\mathbb C}}
\newcommand\bN{{\mathbb N}}
\newcommand\bR{{\mathbb R}}
\newcommand\bZ{{\mathbb Z}}

\newcommand\cD{{\mathcal D}}
\newcommand\cR{{\mathcal R}}

\newcommand\cP{{\mathcal P}}
\newcommand\cQ{{\mathcal Q}}

\newcommand\sD{{\mathscr D}}

\newcommand\sK{{\mathscr K}}
\newcommand\sR{{\mathscr R}}
\newcommand\sX{{\mathscr X}}

\newcommand\la{\lambda}

\newcommand\bg{\mathbf{g}}

\newcommand\bu{\mathbf{u}}
\newcommand\bv{\mathbf{v}}
\newcommand\bla{\bm{\lambda}}
\newcommand\bmu{\bm{\mu}}

\begin{document}

\title{Reconstruction of energy-dependent Sturm--Liouville equations from two spectra}

\author[N.~Pronska]{Nataliya Pronska}%

\address[N.P.]{Institute for Applied Problems of Mechanics and Mathematics,
3b~Naukova st., 79601 Lviv, Ukraine}
\email{nataliya.pronska@gmail.com}

\subjclass[2010]{Primary 34A55, Secondary 34B07, 34B24, 34B30, 34L40, 47E05}%

\keywords{Inverse spectral problem, energy-dependent potentials, Sturm--Liouville operators}%

\date{\today}

\begin{abstract}
In this paper we  study the inverse spectral problem of
reconstructing energy-dependent Sturm--Liouville equations from
two spectra. We give a reconstruction algorithm  and establish
existence and uniqueness of reconstruction. Our approach
essentially exploits the connection between the spectral problems
under study and those for Dirac operators of a special form.
\end{abstract}

\maketitle


\section{Introduction}
In this paper, we consider energy-dependent Sturm--Liouville
 differential equations
\begin{equation}\label{eq:intr.spr}
    -y''+qy+2\lambda p y=\lambda^2y
\end{equation}
on~$(0,1)$; here~$\la\in\bC$ is the spectral parameter,~$p$ is a
real-valued function from~$L_2(0,1)$ and~$q$ is a real-valued
distribution from the Sobolev space~$W_2^{-1}(0,1)$, i.e.~$q =r'$
with a real-valued~$r\in L_{2}(0,1)$ (see a detailed definition in
the next section). Our aim is to study the inverse problem of
reconstructing the potential~$p$ and a primitive~$r$ of~$q$ from
the spectra of the problem~\eqref{eq:intr.spr} under two types of
boundary conditions: the Dirichlet ones
\begin{equation}\label{eq:intr.bcD}
    y(0)=y(1)=0
\end{equation}
and the so-called mixed conditions
\begin{equation}\label{eq:intr.bcM}
    y(0)=(y'-ry)(1)=0.
\end{equation}

The spectral problems under consideration often arise in classical
and quantum mechanics. The most common example is modelling of the
motion of massless particles such as photons by means of the
Klein--Gordon equations, which can be transformed to
Sturm--Liouville equations with potentials depending on the
spectral parameter. The corresponding evolution equations are used
to describe interactions of colliding spinless particles. The
equations of interest also arise in modelling of mechanical
systems vibrations in viscous media, see~\cite{Yam:90}. As the
spectral parameter is related to the energy of the system, this
explains the terminology ``energy-dependent potentials'' widely
accepted in the physical and mathematical literature.

Since equation~\eqref{eq:intr.spr} depends nonlinearly on~$\la$,
we should regard the spectral problems~\eqref{eq:intr.spr},
\eqref{eq:intr.bcD} and~\eqref{eq:intr.spr}, \eqref{eq:intr.bcM}
as those for quadratic operator pencils. Some spectral properties
of such problems can be derived from the general spectral theory
of polynomial operator pencils~\cite{Mar:88}.

An interesting approach to the spectral analysis of Klein--Gordon
equations, using the theory of Krein spaces (i.e.\ spaces with
indefinite scalar products) was suggested by
P.~Jonas~\cite{Jon:93} and H.~Langer, B.~Najman, and
C.~Tretter~\cite{LanNajTre:06,Naj:83,LanNajTre:08}. This method
was also applied to the spectral analysis of problems of interest,
cf.~\cite{Pro:2011a}.

Equation~\eqref{eq:intr.spr} has appeared earlier in the context
of inverse scattering.
 For example,  M.~Jaulent and C.~Jean
studied inverse scattering problems for energy-dependent
Schr\"odinger operators on the line and half-line in~\cite{Jau72, JauJea72,
JauJea761,JauJea762}. Such problems were also considered in
\cite{MeePiv01,SatSzm95, AktMee91, Kam081, MakGus89,MakGus86,
Nab06,NabGus06, Tsu81}.

However, the inverse problems for energy-dependent
Sturm--Liouville equations on a finite interval have
 not been studied sufficiently well. M.~Gasymov and G.~Guseinov
 discussed such problems with~$p\in W_2^1(0,1)$ and~$q\in L_2(0,1)$ and with Robin boundary
conditions in their short paper~\cite{GasGus81} of 1981 containing
no proofs. In the papers \cite{Gus86,GusNab07,Nab04, Nab07,
YanGuo11}  the spectral problems of interest with (quasi)-periodic
boundary conditions were considered, but  only the Borg-type
uniqueness results were obtained therein.

In the present paper, we consider the inverse spectral problem
for~\eqref{eq:intr.spr} under minimal smoothness assumptions on
the real-valued potentials~$p$ and~$q$, including, e.g.\ the case
when~$q$ contains Dirac delta-functions and/or Coulumb-like
singularities. Such functions often arise in modelling of
interactions in molecules and atoms in quantum mechanics, see the
monographs by S.~Albeverio et al.~\cite{AGHH} and by S.~Albeverio
and P.~Kurasov~\cite{AlbKur:99} and the extensive reference lists
therein. The inverse spectral problem for
equation~\eqref{eq:intr.spr} under such minimal smoothness
assumptions on the potentials~$p$ and~$q$ was considered
in~\cite{HryPro:2012}, where the reconstruction from the Dirichlet
spectrum and the set of suitably defined norming constants was
investigated.

Here, we study the problem of reconstruction of the potentials $p$
and $q$ of~\eqref{eq:intr.spr} from the two spectra determined
respectively by the boundary conditions~\eqref{eq:intr.bcD}
and~\eqref{eq:intr.bcM}. Namely, we give a complete description of
these spectra; establish constructively existence of the pencil
with given spectra and thus derive a reconstructing algorithm;
finally, we prove uniqueness of the reconstruction. Note
that~\eqref{eq:intr.bcM} contains an unknown primitive $r$ of $q$;
therefore it is in fact this primitive $r$ (called the regularized
potential, or simply the potential) rather than just $q$ that
should be found. Our approach uses the technique analogous to that
of~\cite{HryPro:2012}. Namely, we shall associate the spectral
problems~\eqref{eq:intr.spr},~\eqref{eq:intr.bcD}
and~\eqref{eq:intr.spr},~\eqref{eq:intr.bcM} with those for Dirac
operators of a special form and shall essentially use the
well-developed inverse spectral theory for Dirac operators in our
analysis.

The paper is organized as follows. In the next section we give the
explicit formulation of the inverse problem of interest and state
the main results. In Section~\ref{sec:Dir}, we show that the
spectral problems under consideration can be reduced to those for
Dirac operators with potentials of a special form. In
Section~\ref{sec:transf}, we introduce the so-called
transformation operators relating such Dirac operators to the
Dirac operators in a canonical form. Existence and uniqueness of
Dirac operators with potential of the special form having given
sets as their spectra is established in Section~\ref{sec:Dir-rec};
in this section we also prove the main results. In the last
section we formulate the reconstruction algorithm and discuss some
possible extensions.

\emph{Notations.} Throughout the paper, we denote by
$L_{2,\mathbb{R}}(0,1)$ and $W_{2,\mathbb{R}}^{-1}(0,1)$ the
spaces of real-valued functions in~$L_2(0,1)$ and distributions in
$W_{2}^{-1}(0,1)$, respectively, and by
$\mathcal{M}_2=\mathcal{M}_2(\bC)$ the linear space of $2\times 2$
matrices with complex entries endowed with the Euclidean operator
norm. Next, $\rho(T)$ and $\sigma(T)$ will stand for the resolvent
set and the spectrum of a linear operator or a quadratic operator
pencil~$T$. The superscript~$\mathrm{t}$ will signify
 the transposition of vectors and matrices, e.g.\
$(c_1,c_2)^{\mathrm{t}}$ is the column vector~$\binom{c_1}{c_2}$.


\section{Preliminaries and main results} \label{sec:pre}


As was mentioned before, the spectral
problems~\eqref{eq:intr.spr}, \eqref{eq:intr.bcD}
and~\eqref{eq:intr.spr}, \eqref{eq:intr.bcM} are the spectral
problems for quadratic operator pencils. To give a precise
definition, we consider first the differential expression
\[
    \ell(y) := - y'' + q y.
\]
Recall that~$q=r'$ is a real-valued distribution in the Sobolev
space~$W_2^{-1}(0,1)$, and to define~$\ell$ rigorously one can
use, e.g.\ the regularization by quasi-derivative method due to
Savchuk and Shkalikov~\cite{SavShk:1999,SavShk:2003}. Following
that method, one introduces the \emph{quasi-derivative}
$y^{[1]}:=y'-ry$ of an absolutely continuous~$y$ and then
define~$\ell$ as
 \[
   \ell (y) = -\bigl(y^{[1]}\bigr)' - r y^{[1]} - r^2 y
\]
on the set
\[
    \dom \ell = \{y \in AC (0,1) \mid y^{[1]} \in AC[0,1], \ \ell(y) \in L_2(0,1)\}.
\]
It is easy to verify that~$\ell (y) = - y'' + qy$ in the sense of
distributions and that for regular potentials~$q$ the above
definition coincides with the standard one.


Denote by~$A_1$ and~$A_2$ Sturm--Liouville operators defined via
\[
    A_j y = \ell(y),\quad j=1,2,
\]
on the domains
\[
    \dom A_1:=\{y \in \dom \ell \mid
        y(0)=y^{[1]}(1) =0\}
\]
and
\[
    \dom A_2:=\{y \in \dom \ell \mid
        y(0)=y(1) =0\}.
\]
For $q\in W_{2,\bR}^{-1}(0,1)$ the operators $A_1$ and~$A_2$ are
known~\cite{SavShk:1999,SavShk:2003} to be self-adjoint, bounded
below, and to have simple discrete spectra.

 Further, we denote by $B$ the operator of
multiplication by the potential~$p\in L_2(0,1)$. The operator~$B$
is in general unbounded. However, since the Green functions of the
operators~$A_j$ are continuous on the square $[0,1]\times[0,1]$,
the resolvents of~$A_j$ are of Hilbert--Schmidt class. It follows
that for every non-real~$\la$ the operators~$B(A-\la)^{-1}$,
$j=1,2$, are also integral operators of Hilbert--Schmidt class; in
particular, $\dom B \supset \dom A_j$, ~$B$
is~$A_j$-compact~\cite[Ch.~IV]{Kat:1966} and thus~$A_j$-bounded
with relative~$A_j$-bound 0~\cite[Lemma~III.2.16]{EngNag:2000}. As
usual, $I$ shall stand for the identity operator in~$L_2(0,1)$.

Now the spectral problems~\eqref{eq:intr.spr}, \eqref{eq:intr.bcD}
and~\eqref{eq:intr.spr}, \eqref{eq:intr.bcM} can be regarded as
the spectral problems for the \emph{quadratic operator pencils}
$T_2(p,r)$ and~$T_1(p,r)$ respectively, given by
\[
   T_j(p,r)(\la):=\la^2I - 2 \la B - A_j, \quad j=1,2,
\]
for $\lambda\in \bC$ on the $\la$-independent domain~$\dom
T_j(p,r):= \dom A_j$. From the above-listed properties of the
operators $A_j$ and $B$ we conclude that for every
$\lambda\in\bC$, the operator~$T_j(p,r)(\lambda)$ is well defined
and closed on $\dom T_j(p,r)$.

 We now recall the following notions
of the spectral theory of operator pencils, see~\cite{Mar:88}.

\begin{definition} An \emph{operator pencil}~$T$ is an
operator-valued function on~$\bC$.
 The \emph{spectrum}~$\sigma(T)$
of an operator pencil~$T$ is the set of all $\lambda\in\bC$ for
which~$T(\lambda)$ is not boundedly invertible, i.e.\
\[
    \sigma(T)=\{\lambda\in\mathbb{C}\mid 0\in\sigma(T(\lambda))\}.
\]
A number $\lambda\in\bC$ is called the \emph{eigenvalue} of $T$ if
$T(\lambda)y=0$ for some non-zero function~$y\in\dom T$, which is
then the corresponding \emph{eigenfunction}. Finally,
\[
    \rho(T):=\mathbb{C}\setminus\sigma(T)
\]
is the \emph{resolvent set}~ of an operator pencil~$T$.
\end{definition}

It was shown in~\cite{Pro:2011a} that the spectra of the operator
pencils~$T_j(p,r)$, $j=1,2$, consist entirely of eigenvalues and
that~$\sigma(T_j(p,r))$ are discrete subsets of~$\bC$. There are
examples of~$T_j(p,r)$ having non-real and/or non-simple
eigenvalues; the latter means that the algebraic multiplicity of
an eigenvalue can be greater than~$1$~\cite{Mar:88}. In such a
generic setting, the inverse spectral problem of reconstructing
the potentials~$p$ and~$q$ in~\eqref{eq:intr.spr} is very
complicated, and only partial existence results have been
established so far. The approach we are going to use gives a
complete solution in the special case where the spectra
of~$T_1(p,r)$ and~$T_2(p,r)$ are real and simple. A sufficient
condition for this to hold is, e.g., that there is a $\mu_*\in\bR$
such that the operator~$T_1(p,r)(\mu_*)$ is
negative~\cite{Pro:2011a}. Thus our standing assumption is that
\begin{itemize}
    \item [(A)] the operator pencil~$T_1(p,r)$ is hyperbolic~\cite[Ch.~IV, \S 31]{Mar:88},
    or,
    equivalently, there is a~$\mu_*\in \bR$ such that the operator~$T_1(p,r)(\mu_*)$ is
negative.
\end{itemize}

Under assumption~(A), the eigenvalues~$\mu_n$,~$n\in\bZ$,
of~$T_1(p,r)$ and~$\la_n$, $n\in\bZ^*:=\bZ\setminus\{0\}$,
of~$T_2(p,r)$ are all real, simple (see~\cite{Pro:2011a}) and can
be labelled in increasing order  so that~$\mu_n$ and~$\la_n$ obey
the asymptotics
\begin{equation*}
    \mu_n = \pi \left(n-\tfrac12\right) + p_0 + \tilde \mu_n, \quad
    \la_n = \pi n + p_0 + \tilde \la_n,
\end{equation*}
with $p_0:=\int_0^1 p(x)\,dx$ and $(\tilde \mu_n), (\tilde
\lambda_n)$ in $\ell_2$ (see \cite{Pro:2012}) and almost interlace
in the sense that~$\mu_{k}<\la_k<\mu_{k+1}$ when~$k\in \bZ^*$ (see
\cite{HryPro:2012p} and Section~\ref{sec:Dir}). Thus the pair of
spectra~$((\la_n),(\mu_n))$ forms then an element of the set~$SD$
defined as follows.

\begin{definition}
\label{def:SD}
 We denote by~$SD$ the family of all pairs~$(\bla, \bmu)$
 of increasing sequences~$\bla:=(\la_n)_{n\in\bZ^*}$ and~$\bmu:=(\mu_n)_{n\in\bZ}$ of real
numbers, which satisfy the following conditions:
\begin{itemize}
    \item [(i)] \emph{asymptotics:} there is an~$h\in\bR$ such that
\begin{equation}
\label{eq:pre.asy}
    \la_n = \pi n + h + \tilde \la_n,\quad
    \mu_n = \pi \left(n-\tfrac{1}{2}\right) + h + \tilde \mu_n,
\end{equation}
where $(\tilde\la_n)$ is a sequence in~$\ell_2(\bZ^*)$
and~$(\tilde\mu_n)$ is from~$\ell_2(\bZ)$;
    \item [(ii)] \emph{almost interlacing:}
    \begin{equation}
\label{eq:pre.aic}
    \mu_{k}<\la_k<\mu_{k+1}\quad \text{for every } k\in\bZ^{*}.
\end{equation}
\end{itemize}
\end{definition}

Take an arbitrary~$\la_0\in(\mu_0,\mu_1)$ and denote by~$\bla^*$
the sequence~$\bla$ augmented with~$\la_0$. Then the almost
interlacing condition~\eqref{eq:pre.aic} means that the
sequences~$\bla^*$ and~$\bmu$ strictly interlace.

\begin{remark}\label{rem:mu} Assume~(A) and let the enumeration of
the eigenvalues of~$T_1(p,r)$ and~$T_2(p,r)$ agree
with~\eqref{eq:pre.aic}; then the number~$\mu_*$ in assumption~(A)
belongs to~$(\mu_0,\mu_1)$. Moreover, the operator~$T_1(p,r)(\la)$
is negative for every~$\la\in(\mu_0,\mu_1)$, so that~(A) holds
with every~$\mu_*\in(\mu_0,\mu_1)$. The same holds
for~$T_2(p,r)(\mu)$; in particular, then $T_2(p,r)(\lambda)$ is
negative for every $\lambda\in (\mu_0,\mu_1)$.
\end{remark}

Denote by~$\bmu$ and~$\bla$ the spectra of the operator
pencils~$T_1(p,r)$ and~$T_2(p,r)$ respectively. Then the inverse
spectral problem of interest is to reconstruct the operator
pencils~$T_j(p,r)$,~$j=1,2$, given the spectra~$\bmu$ and~$\bla$.
Namely, we want to reconstruct the potential~$p$ and the
(regularized) potential~$r$ (determining both the potential~$q=r'$
and the right-end boundary condition for the operator
pencil~$T_1(p,r)$ in~\eqref{eq:intr.bcM}). Some properties of the
sets $\bmu$ and~$\bla$ are given in~\cite{Pro:2011a}; our aim here
is to give a complete characterization of these spectral data and
to find an algorithm reconstructing the potentials $p$ and $r$
from the spectral data.

It is easy to see that with~$p\equiv 0$, the spectral problems for
the operator pencils~$T_j(0,r)$ become the usual spectral problems
$A_jy = \la^2y$ for the Sturm--Liouville operators~$A_j$ with
potential~$q$ in~$ W_2^{-1}(0,1)$;
see~\cite{SavShk:1999,SavShk:2003}. In this case parts of the
sequences $\bmu$ and $\bla$ are redundant in the sense that
$\mu_{1-n}=-\mu_n$ and $\lambda_{-n}=-\lambda_n$.
In~\cite{HryMyk:2004} the authors proved that, just as in the
regular case when $q$ is integrable, the
spectra~$(\mu_n^2)_{n\in\bN}$ of~$A_1$ and~$(\la_n^2)_{n\in\bN}$
of~$A_2$ uniquely determine the regularized potential~$r$; see
also~\cite{SavShk:2005} for an alternative treatment. These papers
also suggest the algorithm of reconstructing the potential~$r$
from the spectra of~$A_1$ and~$A_2$.

In the inverse spectral problem of interest we want to determine
two real-valued potentials $p$ and $r$ of the operator
pencils~$T_j(p,r)$,~$j=1,2$. Since the information contained in
the spectra of~$T_j(p,r)$ is twice as large as for the standard
Sturm--Liouville operators one may hope that the inverse spectral
problem of reconstructing $p$ and $r$ from the spectral data
of~$T_j(p,r)$ is well posed.

As was already mentioned, the spectra of the operator
pencils~$T_j(p,r)$, $j=1,2$, form an element of~$SD$. Our main
result states that, conversely, every element of the set~$SD$
coincides with the spectral data for some operator
pencils~$T_j(p,r)$ under consideration and that these operator
pencils are uniquely determined by their spectra.

\begin{theorem}\label{thm:pre.main}
Assume that a pair~$(\bla,\bmu)$ of sequences of real numbers is
an element of~$SD$. Then there exist unique~$p,r\in
L_{2,\mathbb{R}}(0,1)$ such that~$\bmu$ and~$\bla$ are the spectra
of the pencils~$T_1(p,r)$ and~$T_2(p,r)$ respectively. Moreover,
the operator pencil~$T_1(p,r)$ satisfies assumption~(A).
\end{theorem}

The proof of this theorem is constructive and suggests the
reconstruction algorithm determining the potentials~$p$ and $r$
from the sets~$\bla$ and~$\bmu$, see Section~\ref{sec:alg}.

As was already mentioned, our approach consists in reducing the
spectral problems for operator pencils~$T_j(p,r)$, $j=1,2$, to the
spectral problems for Dirac operators in~$L_2(0,1)\times L_2(0,1)$
with appropriate boundary conditions, see Section~\ref{sec:Dir}.
Using a suitable unitary gauge transformation we can reduce these
Dirac operators to those in the ``shifted'' AKNS normal form. For
AKNS Dirac operators, the direct and inverse spectral problems are
well understood,
see~\cite{LevSar:1991,GasDza:1966,AlbHryMk:2005:RJMP}. Using known
methods we  shall reconstruct the AKNS Dirac operators from the
given spectra and then transform them to the Dirac operators
directly associated with the operator pencils of interest keeping
the spectra unchanged. This will give the sought potentials~$p$
and~$r$ in an explicit form, see~\eqref{eq:pq}.


\section{Reduction to the Dirac system}\label{sec:Dir}


In this section we are going to show that under assumption~(A) the
spectral problems for the operator pencils~$T_1(p,r)$
and~$T_2(p,r)$ can be reduced to the spectral problems for special
Dirac operators. The reduction crucially relies on the fact that
the equation~$y''=q_*y$ with~$q_*:=q+2\mu_*p-\mu_*^2$ and~$\mu_*$
from~(A)  has a solution which is strictly positive on~$[0,1]$.

Observe first that this equation can be recast
as~$\ell(y)=\mu_*^2y-2\mu_*py$ and treated as the first order
linear system~$u_1'=ru_1+u_2$,
$u_2'=(2\mu_*p-\mu_*^2-r^2)u_1-ru_2$, where~$u_1=y$
and~$u_2=y^{[1]}$. Therefore for every complex~$a$ and~$b$ it
possesses a unique solution satisfying the conditions~$y(1)=a$
and~$y^{[1]}(1)=b$.  Let~$z$ denote the solution of the
equation~$y''=q_*y$ subject to the conditions
\[
    z(1)=1,\quad z^{[1]}(1)=0.
\]
We are going to show that~$z$ does not vanish on~$[0,1]$.
\begin{lemma}
Under the standing assumption~(A) the function~$z$ is strictly
positive on~$[0,1]$.
\end{lemma}
\begin{proof}
By assumption, the operator~$-T_1(p,r)(\mu_*)$ is uniformly
positive and thus such is the closure~$\mathfrak{t}_1$ of its
quadratic form. For~$y\in \dom T_1(p,r)$ we find by integration by
parts that
\[
    \mathfrak{t}_1[y]=(-T_1(p,r)(\mu_*)y,y)=\int_0^1|y'|^2+\int_0^1(2\mu_*p-\mu_*^2)|y|^2-2\operatorname{Re}(ry',y).
\]
The quadratic form~$\mathfrak{t}_0[y]=\int_0^1|y'|^2$ considered
on~$\dom T_1(p,r)$ is closable and its
closure~$\tilde{\mathfrak{t}}_0$ acts by the same formula on the
domain
\[
    \dom \tilde{\mathfrak{t}}_0 :=\{u\in W_2^1(0,1) \mid u(0)=0\}.
\]
 Since the quadratic
form~$\int_0^1(2\mu_*p-\mu_*^2)|y|^2-2\operatorname{Re}(ry',y)$ is
relatively bounded with respect to the
form~$\tilde{\mathfrak{t}}_0[y]$ with relative bound~$0$
(cf.~\cite{HryMyk:2001}), Theorem~VI.1.33 of~\cite{Kat:1966}
implies that the domain of~$\mathfrak{t}_1$ coincides with $\dom
\tilde{\mathfrak{t}}_0$, i.e.,
\[
    \dom \mathfrak{t}_1=\{u\in W_2^1(0,1)\mid u(0)=0\}.
\]

Now assume that~$z$ has zeros on~$[0,1]$ and denote the largest of
them by~$x_0$. Then the function
\[
    u(x):=\left\{\begin{array}{ll}
            0, & x<x_0 \\
            z(x), & x\ge x_0 \\
                 \end{array}\right.
\]
belongs to the domain of~$\mathfrak{t}_1$ and is not identically
zero. It is straightforward that~$\mathfrak{t}_1[u]=0$, which
contradicts positivity of~$-T_1(p,r)(\mu_*)$. Therefore~$z$ does
not vanish on~$[0,1]$ yielding the assertion of the lemma.
\end{proof}

Setting~$v:=z'/z$, one easily verifies that $q_*=v'+v^2$;
hence~$q_*$ is a Miura potential~\cite{KapPerShuTop:2005} and the
differential expression~$-u''+q_*u$ can be written in the
factorized form, viz.
\begin{equation}\label{eq:Dir.Afact}
   -u''+q_*u  = -\Bigl(\frac{d}{dx}+v\Bigr)\Bigl(\frac{d}{dx}-v\Bigr)u.
\end{equation}
Since~$z$ satisfies~$(z'-rz)(1)=0$, we have that~$(v-r)(1)=0$.

For $\la\ne\mu_*$ consider the functions~$u_2:=y$
and~$u_1:=(y'-vy)/(\la-\mu_*)$. Then equation~\eqref{eq:intr.spr}
can be recast as the following first order system for~$u_1$
and~$u_2$,
\begin{align}\label{eq:Dir.system1}
    u_2' - v u_2 &= (\la-\mu_*) u_1,\\
    -u_1'-v u_1 + 2 p u_2 &= (\la+\mu_*) u_2. \label{eq:Dir.system2}
\end{align}
Setting
\begin{equation}\label{eq:Dir.P}
    J:=\left(
           \begin{array}{cc}
             0 & 1 \\
             -1 & 0 \\
           \end{array}
         \right),
\qquad
    P:=\left(
                \begin{array}{cc}
                  \mu_* & -v \\
                  -v & 2p-\mu_* \\
                \end{array}
              \right),
\qquad
    \bu(x)=\binom{u_1}{u_2},
\end{equation}
we see that the above system is the spectral problem
$\ell(P)\bu=\la\bu$ for a Dirac differential expression $\ell(P)$
acting in $L_2(0,1)\times L_2(0,1)$ via
\begin{equation}
\label{eq:dif.exp.Dirac}
    \ell(P)\bu =J\frac{d\bu}{dx}+P\bu
\end{equation}
on suitable domain.

Denote by~$\sD_j(P)$,~$j=1,2$, the Dirac operators generated
by~$\ell (P)$ on the domains
\[
    \dom\sD_j(P):=\{\bu=(u_1,u_2)^\mathrm{t}\in W_2^1(0,1)\times W_2^1(0,1)
        \mid u_2(0)=u_j(1)=0\}.
\]

It turns out that the spectra of~$\sD_j(P)$ and~$T_j(p,r)$ are
closely related.
\begin{lemma}\label{lem:Dir.spectra}
The spectra of the Dirac operators~$\sD_j(P)$ and the operator
pencils~$T_j(p,r)$, $j=1,2$, are related in the following way:
\begin{align}\label{eq:Dir.coin1}
    \sigma\bigl(\sD_1(P)\bigr) &= \sigma(T_1({p,r})),\\
    \sigma\bigl(\sD_2(P)\bigr) &= \sigma(T_2({p,r})) \cup
    \{\mu_*\}.\label{eq:Dir.coin2}
\end{align}
Moreover,~$y$ is an eigenfunction of~$T_j(p,r)$ corresponding to
an eigenvalue~$\la$ if and only if~$\bu=(u_1,u_2)^\mathrm{t}$
with~$u_1:=(y'-vy)/(\la-\mu_*)$ and~$u_2=y$  is an eigenfunction
of~$\sD_j(P)$ corresponding to~$\la\ne\mu_*$.
\end{lemma}
\begin{proof}
It is straightforward to see that~$\la=\mu_*$ is an eigenvalue of
the Dirac operator~$\sD_2(P)$, the corresponding eigenfunction
being $\bu=(u_1,u_2)^{\mathrm{t}}$ with $u_1 = \exp(-\int v)$ and
$u_2\equiv0$.  However, by Remark~\ref{rem:mu}, under the standing
assumption~(A) the number $\lambda=\mu_*$ is not an eigenvalue
of~$T_2(p,r)$. The proof of~\eqref{eq:Dir.coin2} uses this fact
and is analogous to the proof of the corresponding lemma
in~\cite{HryPro:2012}. The coincidence of the
spectra~$\sigma\bigl(\sD_1(P)\bigr)$ and~$\sigma(T_1(p,r))$ is
proved along the same lines using the fact that the boundary
condition~$y^{[1]}(1)=0$ is equivalent to~$u_1(1)=0$ in view of
the relation~$(v-r)(1)=0$. The last statement of the lemma is
obtained by direct verification.
\end{proof}


\section{Transformation operator}\label{sec:transf}


Lemma~\ref{lem:Dir.spectra} suggests that we can try to use the
spectra~$\bmu$ and~$\bla$ of the operator pencils~$T_1(p,r)$
and~$T_2(p,r)$ in order to find the related Dirac
operators~$\sD_j(P)$,~$j=1,2$. Having determined the
potential~$P=(p_{ij})_{i,j=1}^2$ of $\ell(P)$, we then identify
the potentials~$p$ and $r$ of the operator pencil~$T_j({p,r})$ as
$p:=(p_{22}+p_{11})/2$ and $r: =
-p_{12}-\int_x^{1}(p_{12}^2-p_{11}p_{22})$.

However, the classical inverse spectral theory reconstructs a
Dirac operator with potential in the AKNS form or in other
canonical form. Thus to find the Dirac operators directly
associated with the operator pencils~$T_j(p,r)$, i.e.\ with
potentials of the form~\eqref{eq:Dir.P}, we have to transform the
obtained Dirac operators in canonical form keeping the spectral
data unchanged. This is done by means of the so-called
transformation operators. In~\cite{HryPro:2012} such operators
were constructed and some of their properties were discussed; see
also~\cite{CoxKno:1996,LevSar:1991}. In this section we shall
focus our attention on some further properties of the
transformation operators that are important for our reconstructing
procedure.

Assume that $P$ and $Q$ are $2\times 2$ matrix-valued potentials
in $L_2((0,1),\mathcal{M}_2)$ and set
\[
    \cD_0:=\{(u_1,u_2)^{\mathrm{t}}\in W_2^1(0,1)\times W_2^1(0,1)
    \mid u_2(0)=0\}.
\]
The transformation operator  $\sX=\sX(P,Q)$ between the Dirac
operators~$\ell(P)$ and $\ell(Q)$ on the set~$\cD_0$, i.e.\ the
nontrivial operator satisfying the
relation~$\sX\ell(P)\bu=\ell(Q)\sX\bu$ for all $\bu\in\cD_0$, was
constructed in~\cite{HryPro:2012} in the form
\begin{equation}
\label{eq:Tr.op} \sX\bu(x)=R(x)\bu(x)+\int_0^xK(x,s)\bu(s)ds,
\end{equation}
where $R$ and $K$ are $2\times 2$ matrix-valued functions of one
and two variables respectively. Under the normalization~$R(0)=I$,
the operator $R$ is explicitly given~\cite{CoxKno:1996} as
\begin{equation}
    \label{eq:Dir.R}
    R(x)=e^{\theta_1(x)}\left(%
    \begin{array}{cc}
    \cos\theta_2(x) & \sin\theta_2(x) \\
     -\sin\theta_2(x) & \cos\theta_2(x) \\
    \end{array}%
    \right) = e^{\theta_1(x)I+\theta_2(x)J},
\end{equation}
with
\begin{equation}\label{eq:Dir.theta}
   \begin{aligned}
        \theta_1(x)=\frac{1}{2}\int_0^x\mathrm{tr}[J(Q(s)-P(s))]\,ds,\\
        \theta_2(x)=\frac{1}{2}\int_0^x\mathrm{tr}(Q(s)-P(s))ds.
    \end{aligned}
\end{equation}
Existence of the transformation operator~$\sX(P,Q)$ is guaranteed
by the following two theorems.

\begin{theorem}[\hspace*{-5pt}\cite{HryPro:2012}]\label{thm:Dir.tr-op}
Assume that $P$ and $Q$ are in $L_2((0,1),\mathcal{M}_2)$. Then an
operator~$\sX(P,Q)$ of the form~\eqref{eq:Tr.op}, with~$R$ obeying
the condition~$R(0)=I$ and a summable kernel~$K$, is a
transformation operator for~$\ell(P)$ and~$\ell(Q)$ on the
set~$\cD_0$ if and only if the matrix-valued function~$R$ is given
by~\eqref{eq:Dir.R}--\eqref{eq:Dir.theta} and the kernel~$K$ is a
mild solution of the partial differential equation
\begin{equation}
\label{eq:Dir.K}
    J\partial_xK(x,y)+\partial_yK(x,y)J=K(x,y)P(y)-Q(x)K(x,y)
\end{equation}
in the domain~$\Omega:=\{(x,y)\mid 0<y<x<1\}$ satisfying
for~$0\leq x\leq1$ the boundary conditions
\begin{align}\label{eq:K_b.c_1}
    K(x,x)J-JK(x,x)&=JR'(x)+Q(x)R(x)-R(x)P(x),\\
    K_{12}(x,0)&=K_{22}(x,0)=0.\label{eq:K_b.c_2}
\end{align}
\end{theorem}

\begin{theorem}[\hspace*{-3pt}\cite{HryPro:2012}]
\label{thm:K.exist} Assume that matrix-valued functions $P$ and
$Q$ are in~$L_2((0,1),\mathcal{M}_2)$. Then the
system~\eqref{eq:Dir.K}--\eqref{eq:K_b.c_2} has a unique solution
in the sense of distributions; moreover, this solution belongs
to~$L_2(\Omega,\mathcal{M}_2)$.
\end{theorem}

Throughout the rest of the paper, we shall assume that the
matrix-valued potentials $P$ and $Q$ are Hermitian, i.e.\ that
$P^*(x) = P(x)$ and $Q^*(x) = Q(x)$ a.e.\ on $[0,1]$. Then the
corresponding Dirac operators~$\sD_j(P)$ and $\sD_j(Q)$,~$j=1,2$,
are self-adjoint and have simple discrete spectra. We denote
by~$\bmu(P)$  (resp. by $\bmu(Q)$) the spectrum of~$\sD_1(P)$
(resp. of~$\sD_1(Q)$) and by~$\bla(P)$ (resp. by~$\bla(Q)$) the
spectrum of~$\sD_2(P)$ (resp. of~$\sD_2(Q)$). The sets $\bmu(P)$
and $\bla(P)$ interlace and their elements can be labelled by
$n\in\bZ$ so that~$\mu_n = \pi \left(n-\tfrac12\right) +
\tfrac12\int_0^1\tr P + \mathrm{o(1)}$ and~$\la_n = \pi n +
\tfrac12\int_0^1\tr P + \mathrm{o(1)}$ as
$|n|\to\infty$~(see~\cite{LevSar:1991}); the spectra $\bmu(Q)$ and
$\bla(Q)$ have similar properties.

 Also,
$\sX=\sX(P,Q)$ will stand for the transformation operator of the
form~\eqref{eq:Tr.op} for the differential expressions~$\ell(P)$
and $\ell(Q)$ on the domain~$\cD_0$, with $R$ given
by~\eqref{eq:Dir.R} and \eqref{eq:Dir.theta}. We shall write
$\sX=\sR + \sK$, where $\sR\bu(x):=R(x)\bu(x)$ is the operator of
multiplication by $R$ and
\[
    \sK\bu(x):= \int_0^x K(x,s)\bu(s)\,ds
\]
is the corresponding integral operator. We observe that for
Hermitian~$P$ and $Q$ the functions~$i \theta_1$ and $\theta_2$
are real valued; in particular, the operator~$\sR$ is unitary.

\begin{remark}\label{rem:Dir.X}
Let us note that if $u\in \cD_0$, then the relation
\[
    \bigl(\ell(P)-\lambda\bigr)\bu = \mathbf{f}
\]
holds if and only if for $\bv:=\sX\bu$ and $\bg:=\sX\mathbf{f}$
one gets
\[
    \bigl(\ell(Q) - \lambda\bigr)\bv = \bg.
\]
\end{remark}

\begin{theorem}
\label{thm:sp.data.coinc.crit} Assume that the matrix
potentials~$P$ and~$Q$ are Hermitian. Then the spectra of the
operators~$\sD_j(P)$ and~$\sD_j(Q)$,~$j=1,2$, coincide, i.e.\
$\bmu(P)=\bmu(Q)$ and~$\bla(P)=\bla(Q)$, if and only if the
transformation operator~$\sX(P,Q)$ for $\ell(P)$ and $\ell(Q)$ on
the domain~$\cD_0$ only contains the unitary part $\sR$ (i.e.\
$\sK=0$) and~$\theta_2(1)=\pi n$ for some $n\in\bZ$.
\end{theorem}

\begin{proof}
\emph{Sufficiency}.  By assumption, the transformation
operator~$\sX(P,Q)$ between~$\ell(P)$ and~$\ell(Q)$ only contains
the multiplication operator~$\sR$. Denote
by~$\widetilde{\sD}_j(P)$,~$j=1,2$, the operators related
to~$\sD_j(Q)$ by
\[
    \widetilde{\sD}_j(P)=\sR^{-1}\sD_j(Q)\sR,\; j=1,2.
\]
Then~$\widetilde{\sD}_j(P)$ are unitarily equivalent to~$\sD_j(Q)$
and so $\sigma(\widetilde{\sD}_j(P))=\sigma(\sD_j(Q))$. By
Remark~\ref{rem:Dir.X}, $\widetilde{\sD}_j(P)$ are Dirac operators
acting via~$\widetilde{\sD}_j(P)\bu=\ell(P)\bu$  on  the domains
which consist of those~$\bu\in L_2((0,1),\bC^2)$ for
which~$\sR\bu\in\dom \sD_j(Q)$.
 Recall that~$R(0)=I$; also since~$\theta_2(1)=\pi n$ for
some~$n\in \bZ$, we have that~$R(1)$ is a multiple of the identity
matrix~$I$. Therefore these~$\bu$ satisfy the same boundary
conditions as~$\sR \bu$.
Thus~$\dom\widetilde{\sD}_j(P)=\dom\sD_j(Q)=\dom\sD_j(P)$. This
means that~$\sD_j(P)=\widetilde{\sD}_j(P)$, and
so~$\bmu(P)=\bmu(Q)$ and~$\bla(P)=\bla(Q)$.

\emph{Necessity}. Firstly note that as the spectra of~$\sD_j(P)$
and~$\sD_j(Q)$, $j=1,2$, coincide, the asymptotics of the
eigenvalues also coincide thus yielding~$\theta_2(1)=\pi n$ for
some~$n\in \bZ$. To complete the proof it remains to show
that~$\sK=0$.

Recall that the norming constant corresponding to an
eigenvalue~$\la$ of the operator~$\sD_2(P)$ is defined
as~$\|\bu\|^2$, where~$\bu=(u_1,u_2)^\mathrm{t}$ is the
eigenfunction of~$\sD_2(P)$ for~$\la$ normalized by the initial
conditions $u_1(0)=1$ and $u_2(0)=0$~\cite{LevSar:1991}. It
follows from~\cite{AlbHryMk:2005:RJMP} that the norming constants
for~$\sD_2(P)$ are uniquely determined by the spectra~$\bla(P)$
and~$\bmu(P)$ (although only AKNS potentials were treated
in~\cite{AlbHryMk:2005:RJMP}, the arguments therein are valid for
all self-adjoint Dirac operators). Thus from the statement of the
theorem we conclude that the sets of norming constants
for~$\sD_2(P)$ and~$\sD_2(Q)$ coincide. By Lemma 4.4.
of~\cite{HryPro:2012} this gives that $\sK=0$.
\end{proof}

Alternatively, to establish the necessity part of the above
theorem, one can use arguments analogous to those
in~\cite{CoxKno:1996}, where a similar statement but under
different assumptions on~$P$ and~$Q$ was proved.


\section{Reconstruction of the pencils}\label{sec:Dir-rec}


In this section, we study the problem of reconstructing the
potentials~$p$ and~$r$ of the corresponding operator pencils
$T_j({p,r})$,~$j=1,2$, from their spectra. Our aim is to prove
Theorem~\ref{thm:pre.main}, i.e.\ that given an arbitrary element
$(\bla,\bmu)$ of $SD$, there exist unique $p$ and $r$ such
that~$\bmu$ is the spectrum of the operator pencil~$T_1(p,r)$
and~$\bla$ is that of~$T_2(p,r)$.

Fix therefore an arbitrary pair~$(\bla,\bmu)$ in SD. In what
follows,~$\la_n$,~$n\in\bZ^*$, and~$\mu_n$,~$n\in\bZ$, will stand
respectively for elements of~$\bla$ and~$\bmu$. Note that the
enumeration of~$\mu_n$ and~$\la_n$ fixes the shift~$h$ in their
asymptotics~\eqref{eq:pre.asy}. We set~$\mu_*$ to be the middle
point of the interval~$(\mu_0,\mu_1)$,
i.e.~$\mu_*:=(\mu_0+\mu_1)/2$, and augment the
sequence~$(\la_n)_{n\in\bZ^*}$ with an element~$\la_0=\mu_*$;
denote obtained sequence by~$\bla^*$.

Recall now some facts from the inverse spectral theory for Dirac
operators, which we shall use in our procedure,
see~\cite{AlbHryMk:2005:RJMP,GasDza:1966,LevSar:1991}. Consider
the set~$\cQ_0$ of~$2\times2$ matrix-valued functions of the AKNS
normal form, namely
\begin{equation}\label{eq:ex.Q0}
    \cQ_0:=\left\{ Q_0=
    \begin{pmatrix}
        q_1 & q_2 \\ q_2 & -q_1
    \end{pmatrix} \mid q_j \in L_{2,\bR}(0,1)\right\}.
\end{equation}
It is known that the operators~$\sD_1(Q_0)$ and~$\sD_2(Q_0)$ with
potential~$Q_0$ from~$\cQ_0$ are self-adjoint and have simple
discrete spectra. Their eigenvalues can be enumerated in
increasing order as $\la_n(Q_0)$ and~$\mu_n(Q_0)$, $n\in\bZ$,
respectively so that they satisfy the interlacing condition
\[
    \mu_n(Q_0)<\la_n(Q_0)<\mu_{n+1}(Q_0)
\]
and obey the asymptotics
\begin{align*}
\la_n(Q_0)& = \pi n + \tilde \la_n(Q_0),\\
\mu_n(Q_0)& = \pi \left(n-\tfrac{1}{2}\right)+ \tilde \mu_n(Q_0)
\end{align*}
with $\ell_2(\bZ)$-sequences~$(\tilde \la_n(Q_0))$ and~$(\tilde
\mu_n(Q_0))$.

It is also known (see~\cite{AlbHryMk:2005:RJMP}) that for two
sequences~$(\mu_n)$ and~$(\la_n)$ having the above properties
there exists a unique potential~$Q_0$ from~$\cQ_0$ such that the
spectrum of~$\sD_1(Q_0)$ is~$(\mu_n)$ and that of~$\sD_2(Q_0)$
is~$(\la_n)$.

The above sequences~$\bmu$ and~$\bla^*$ differ from the spectra
for Dirac operators with potentials of normal AKNS form only by
the shift~$h$ in their asymptotics, see~\eqref{eq:pre.asy}. For
$h\in\bR$, we denote by
\[
    \cQ_h:= \{Q_0+hI \mid Q_0 \in \cQ_0\}
\]
the set of $h$-\emph{shifted} AKNS matrix potentials; then the
following result holds true.

\begin{proposition}\label{pro:exist.Q}
For an arbitrary pair~$(\bla,\bmu)$ from~$SD$
take~$\mu_*:=(\mu_0+\mu_1)/2$  and denote by $\bla^*$ the
augmentation of $\bla$ with $\la_0=\mu_*$. Then there exists a
unique potential $Q\in\cQ_h$ such that $\bmu$ is the spectrum of
the operator~$\sD_1(Q)$ and~$\bla^*$ is that of~$\sD_2(Q)$.
\end{proposition}

For an arbitrary $P\in L_2\bigl((0,1),\mathcal{M}_2\bigr)$, we
introduce the set
\[
    \Iso(P):=\{\tilde P\in L_2\bigl((0,1),\mathcal{M}_2\bigr) \mid
    \bmu(\tilde P)=\bmu(P),\; \bla(\tilde P)=\bla(P)\}
\]
of \emph{isospectral potentials} and for~$\mu\in\bR$ denote
by~$\cP_{\mu}$ the set of all potentials of the
form~\eqref{eq:Dir.P}, i.e.\
\[
    \mathcal{P}_\mu:=\{P=(p_{ij})_{i,j=1}^2 \mid
        p_{ij}\in L_{2,\mathbb{R}}(0,1),
        \ p_{11}=\mu,\ p_{12}=p_{21}\}.
\]

Now we are going to prove that for every~$\mu\in\bR$
and~$Q\in\cQ_h$ there exists a unique potential~$P$
from~$\Iso(Q)\cap \mathcal{P}_\mu$. We start with the following
lemma.

\begin{lemma}\label{lem:QRP_exi_uni}
Suppose that~$Q \in \cQ_h$,~i.e.\
\[
    Q=\left(%
    \begin{array}{cc}
    q_1+h & q_2 \\
    q_2 & -q_1+h \\
    \end{array}%
    \right).
\]
Then for every fixed~$\mu\in\bR$ there exists a unique
potential~$P\in\cP_\mu$ such that, with $\mathcal{R}$ being the
operator of multiplication by the matrix-valued function $R$
of~\eqref{eq:Dir.R}--\eqref{eq:Dir.theta}, one has
\begin{equation}\label{eq:Rec.RPQ}
    \sR\ell(P)=\ell(Q)\sR.
\end{equation}
\end{lemma}

\begin{proof}
 First we note that the matrix~$R$
given by~\eqref{eq:Dir.R} commutes with~$J$, whence
\[
  R^{-1}\Bigl(J\frac{d}{dx}+Q\Bigr)R
        =J\frac{d}{dx}+R^{-1}JR'+R^{-1}QR.
\]
Thus the relation~\eqref{eq:Rec.RPQ} requires that
\begin{equation}\label{eq:Rec.P}
    \begin{aligned}
     P  & = R^{-1}JR'+R^{-1}QR \\
        & = (h-\theta_2')I+\theta_1'J
          + \left(
            \begin{array}{cc}
                q_{1}\cos2\theta_2 - q_{2}\sin 2\theta_2
              & q_{1}\sin2\theta_2 + q_{2}\cos 2\theta_2 \\
                q_{1}\sin 2\theta_2 + q_{2}\cos2\theta_2
              &-q_{1}\cos2\theta_2 + q_{2}\sin 2\theta_2
            \end{array}%
            \right).
    \end{aligned}
\end{equation}
Now we fix an arbitrary~$\mu\in\bR$ and observe that the
potential~$P$ of~$\eqref{eq:Rec.P}$ belongs to~$\cP_\mu$ if and
only if~$\theta_1(x)$ is identically zero and the following
equality holds:
\begin{equation}\label{eq:Rec.theta}
    -\theta_2'+q_{1}\cos2\theta_2-q_{2}\sin 2\theta_2+h=\mu.
\end{equation}
Clearly, there exists a unique solution of the above equation
satisfying the initial condition
\begin{equation}\label{eq:Rec.theta.ic}
\theta_2(0)=0.
\end{equation}
This solution and~$\theta_1\equiv
 0$ verify the relation~\eqref{eq:Dir.theta} with the given~$Q$
 and with~$P$ of~\eqref{eq:Rec.P}. The potential~$P$ is
 explicitly given by~\eqref{eq:Rec.P} with~$\theta_1\equiv 0$ and~$\theta_2$ as above and by construction it belongs
 to~$\cP_\mu$ and satisfies~\eqref{eq:Rec.RPQ}. The proof is complete.
\end{proof}

The next result states existence and uniqueness of the
potential~$P$ from~$\cP_\mu$ belonging to~$\Iso(Q)$.

\begin{theorem}\label{th:exist_uniq_P}
Let~$\mu,h \in\bR$ and assume that~$Q \in \cQ_h$ is such that
$\la=\mu$ is an eigenvalue of the Dirac operator~$\sD_2(Q)$. Then
there exists a unique~$P\in\mathcal{P}_\mu$ belonging
to~$\Iso(Q)$, i.e.\ \(
    \Iso(Q)\cap\mathcal{P}_\mu=\{P\}.
\)
\end{theorem}
\begin{proof}
Observe first that if~$P\in \Iso(Q)$, then by
Theorem~\ref{thm:sp.data.coinc.crit} the transformation operator
between~$\ell(P)$ and~$\ell(Q)$ on~$\cD_0$ is just the
operator~$\sR$ of multiplication by the matrix-valued function $R$
given by~\eqref{eq:Dir.R}--\eqref{eq:Dir.theta}. If, moreover,~$P$
should belong to~$\mathcal{P}_\mu$, then by
Lemma~\ref{lem:QRP_exi_uni} $P$ should be given
by~\eqref{eq:Rec.P} with~$\theta_1\equiv0$ and~$\theta_2$
solving~\eqref{eq:Rec.theta}--\eqref{eq:Rec.theta.ic}. It remains
to prove that the potential given
by~\eqref{eq:Rec.P}--\eqref{eq:Rec.theta.ic} indeed belongs
to~$\Iso(Q)$. By Theorem~\ref{thm:sp.data.coinc.crit}, to this end
it suffices to show that~$\theta_2(1)=\pi n$ for some~$n\in
\mathbb{Z}$.

Consider the operator~$\widetilde{\sD}_2(P):=\sR^{-1}\sD_2(Q)\sR$.
Clearly, $\widetilde{\sD}_2(P)$ is a Dirac operator acting
via~$\widetilde{\sD}_2(P)\bu=\ell(P) \bu$ on the domain consisting
of those~$\bu$ for which~$\cR\bu$ belongs to~$\dom \sD_2(Q)$. By
 construction,~$\sigma(\widetilde{\sD}_2(P))=\sigma(\sD_2(Q))$ and thus~$\la=\mu$ is in the spectrum of~$\widetilde{\sD}_2(P)$. Denote
by~$\bu_0=(u_1,u_2)^\mathrm{t}$ the eigenfunction
of~$\widetilde{\sD}_2(P)$ corresponding to~$\la=\mu$. Then
$\bv_0=(v_1,v_2)^\mathrm{t}:=\sR\bu_0$ is an eigenfunction
of~$\sD_2(Q)$ corresponding to~$\la$, so that~$v_2(0)=v_2(1)=0$.
The equality~$\ell(P)\bu_0=\mu\bu_0$ yields the
relation~$u_2'-vu_2=0$. Since~$R(0)=I$, we have $u_2(0)=v_2(0)=0$,
which together with the above relation gives that~$u_2$ is
identically zero. Therefore
\[
    \bv_0(1)
        = R(1)\bu_0(1)
        = \binom{u_1(1)\cos\theta_2(1)}{-u_1(1)\sin\theta_2(1)}.
\]
But~$v_2(1)=0$ and~$u_1(1)\ne 0$ giving that~$\theta_2(1)=\pi n$,
$n\in\mathbb{Z}$.
\end{proof}

With these preliminaries at hand we can prove the main results of
the paper.
\begin{proof}[Proof of Theorem~\ref{thm:pre.main}]
\emph{Existence.} Given a pair~$(\bla,\bmu)$ from~$SD$, we
set~$\mu_*:=(\mu_0+\mu_1)/2$ and augment~$\bla$
with~$\la_0:=\mu_*$ to obtain~$\bla^*$. Using
Proposition~\ref{pro:exist.Q}, we find the potential~$Q$
in~$\cQ_h$ such that~$\bmu$ is the spectrum of the
operator~$\sD_1(Q)$ and~$\bla^*$ is that of~$\sD_2(Q)$.
Theorem~\ref{th:exist_uniq_P} gives that for this~$Q$ we can
find~$P\in\cP_{\mu_*}$ belonging to~$\Iso(Q)$, i.e. such
that~$\bmu$ and~$\bla^*$ are respectively the spectra
of~$\sD_1(P)$ and~$\sD_2(P)$. In view of
Lemma~\ref{lem:Dir.spectra} this means that~$\bmu$ and~$\bla$ are
the spectra of~$T_1(p,r)$ and~$T_2(p,r)$ with
\begin{equation}
\label{eq:pq}
 p=\frac{p_{22}+p_{11}}{2},\quad r: =
-p_{12}-\int_x^{1}\!(p_{12}^2-p_{11}p_{22}).
\end{equation}
The almost interlacing of~$\bla$ and~$\bmu$ gives that~$T_1(p,r)$
is hyperbolic (see~\cite{Pro:2011a}), i.e.\ it satisfies
assumption~(A).

 This completes the proof of existence.

\emph{Uniqueness.} Suppose there are two pairs of
potentials~$p$,~$r$ and~$\hat{p}$,~$\hat{r}$ such
that~$\sigma(T_1(p,r))=\sigma(T_1(\hat{p},\hat{r}))=:\bmu$
and~$\sigma(T_2(p,r))=\sigma(T_2(\hat{p},\hat{r}))=:\bla$. Then
(see \cite{Pro:2011a}) for every~$\mu\in(\mu_0,\mu_1)$ the
operators~$T_1(p,r)(\mu)$ and~$T_1(\hat{p},\hat{r})(\mu)$ are
negative, i.e.\ the assumption~(A) holds for operator
pencils~$T_1(p,r)$ and~$T_1(\hat{p},\hat{r})$.

Set~$\mu_*:=(\mu_0+\mu_1)/2$; then the pencils~$T_j(p,r)$
and~$T_j(\hat{p},\hat{r})$,~$j=1,2$, lead to two Dirac operators
with potentials~$P$ and~$\hat{P}$ in~$\cP_{\mu_*}$ as explained in
Section~\ref{sec:Dir}  and to two Dirac operators with
potentials~$Q$ and~$\hat{Q}$ in~$\cQ_h$, with the same~$h$ as
explained at the beginning of this section. By
construction,~$\bmu(Q)=\bmu(\hat{Q})=\bmu$
and~$\bla(Q)=\bla(\hat{Q})=\bla\cup\{\mu_*\}$, whence~$Q=\hat{Q}$
by Proposition~\ref{pro:exist.Q}. By Lemma~\ref{lem:Dir.spectra}
we have that~$\bmu(P)=\bmu(\hat{P})$ and~$\bla(P)=\bla(\hat{P})$
and thus~$P$ and~$\hat{P}$ belong to the same isospectral
set~$\Iso(Q)$. In view of Theorem~\ref{th:exist_uniq_P} this
yields that~$P=\hat{P}$. Therefore,~$\hat{p}=\hat{p}$
and~$\hat{r}=\hat{r}$, which completes the proof.

\end{proof}


\section{Reconstructing algorithm}\label{sec:alg}


To summarize  let us formulate the algorithm of reconstruction of
the operator pencils~$T_1(p,r)$ and $T_2(p,r)$.

Suppose we have a pair of sequences~$(\bla,\bmu)$ from~$SD$. Then
the reconstructing procedure consists of the following steps:

\begin{enumerate}
    \item Take~$\mu_*:=(\mu_1+\mu_2)/2$ and augment the given sequence~$\bla$ with an
    element~$\lambda_0=\mu_*$;
    \item Use the obtained spectral data to construct the Dirac
    operators $\mathcal{D}_1(Q)$ and $\mathcal{D}_2(Q)$ with the potential~$Q$ of the~$h$-shifted AKNS form;
    \item Find the corresponding potential~$P\in\mathcal{P}_{\mu_*}\cap
    \Iso(Q)$ using~\eqref{eq:Rec.P} and~\eqref{eq:Rec.theta};
    \item Compute the potentials~$p$ and~$r$ using formulae~\eqref{eq:pq}.
\end{enumerate}

We conclude this paper with some remarks. Firstly, note that the
described approach can be used to reconstruct energy-dependent
Sturm--Liouville equations subject to other boundary conditions or
under different smoothness assumptions on the potentials, e.g.\ in
the case when~$p$ and~$r$ are from~$L_s(0,1)$ with~$s\ge1$ or
belong to~$W_2^s(0,1)$ with~$s\ge0$,
cf.~\cite{HryMyk:2006:PEMS,SavShk:2006}.

Secondly, the considered method is also applicable to the inverse
problems of reconstructing the quadratic pencil from different
sets of spectral data, for example, the spectrum and the norming
constants~\cite{HryPro:2012} or Hochstadt--Lieberman mixed data
(cf.~\cite{HocLie:1978}). Finally, one can establish an analogue
of the Hochstadt-type perturbation formulas when finitely many
eigenvalues in one spectrum have been changed
(cf.~\cite{Hoch:73}).

\bibliographystyle{abbrv}

\bibliography{My_bibl}

\end{document}